\documentclass[11pt]{amsart}

\usepackage{amscd}
\usepackage{amsmath, amssymb}
\usepackage{amsfonts}

\usepackage{enumitem,titletoc}

\numberwithin{equation}{section}

\def\XXint#1#2#3{{\setbox0=\hbox{$#1{#2#3}{\int}$}
\vcenter{\hbox{$#2#3$}}\kern-.5\wd0}}

\newcommand{\tr}[2]{\textrm{tr}_{#1} \, {#2}}
\newcommand{\ve}{\varepsilon}

\newcommand{\ocon}{\omega_{\textrm{con}}}

\newcommand{\dbar}{\overline{\partial}}

\newcommand{\ddt}[1]{\frac{\partial #1}{\partial t}}

\newcommand{\ov}[1]{\overline{#1}}
\newcommand{\Xc}{X_{\textrm{can}}}

\newcommand{\ddbar}{\frac{\sqrt{-1}}{2\pi} \partial\dbar}

\newcommand{\oeuc}{\omega_{\textrm{Eucl}}}

\numberwithin{equation}{section}
\newcommand{\oke}{\omega_{\textrm{KE}}}

\begin{document}
\newcounter{remark}
\newcounter{theor}
\setcounter{remark}{0} \setcounter{theor}{1}
\newtheorem{claim}{Claim}
\newtheorem{theorem}{Theorem}[section]
\newtheorem{proposition}{Proposition}[section]
\newtheorem{lemma}{Lemma}[section]
\newtheorem{definition}{Definition}[section]
\newtheorem{conjecture}{Conjecture}[section]
\newtheorem{corollary}{Corollary}[section]
\newenvironment{remark}[1][Remark]{ \begin{trivlist}
\item[\hskip \labelsep {\bfseries #1
\thesection.\theremark}]}{\end{trivlist}}
\newenvironment{example}[1][Example]{\addtocounter{remark}{1} \begin{trivlist}
\item[\hskip \labelsep {\bfseries #1
\thesection.\theremark}]}{\end{trivlist}}
~

\title[The K\"ahler-Ricci flow on surfaces of general type]{Geometric convergence of \\ the K\"ahler-Ricci flow  on\\ complex surfaces of general type}

\author[B. Guo]{Bin Guo}
\thanks{Research supported in part by NSF grants DMS-1406124 and DMS-1406164}
\address{Department of Mathematics, Rutgers University, Piscataway, NJ 08854}
\author[J. Song]{Jian Song}
\address{Department of Mathematics, Rutgers University, Piscataway, NJ 08854}
\author[B. Weinkove]{Ben Weinkove}
\address{Department of Mathematics, Northwestern University, Evanston, IL 60208}

\begin{abstract}   We show that on smooth minimal surfaces of general type, the K\"ahler-Ricci flow starting at any initial K\"ahler metric converges in the Gromov-Hausdorff sense to a K\"ahler-Einstein orbifold surface. In particular, the diameter of the evolving metrics is uniformly bounded for all time and the K\"ahler-Ricci flow contracts all the holomorphic spheres with $(-2)$ self-intersection number to isolated orbifold points.  Our estimates do not require a priori the existence of an orbifold K\"ahler-Einstein metric on the canonical model. 

 \end{abstract}

\maketitle

\section{Introduction}

The Ricci flow, introduced by Hamilton \cite{H}, has become a powerful tool to study the topology and geometric structures of  Riemannian manifolds. In general, the Ricci flow develops finite time singularities. Hamilton's program of \emph{Ricci flow with surgeries} was carried out by Perelman \cite{P1, P2, P3} to prove Thurston's geometrization conjecture. The minimal model theory in birational geometry can be viewed as the complex analogue of Thurston's geometrization conjecture.  It was proposed in \cite{SoT3} that the Ricci flow will carry out an \emph{analytic minimal model program} on projective varieties, analogous to the 3-dimensional Ricci flow with surgeries.  

In particular, it was 
 conjectured in \cite{SoT3} that the Ricci flow will perform a canonical surgery, corresponding to a birational transformation such as a divisorial contraction or flip, whenever it encounters a non-volume-collapsing singularity at finite time (see also \cite{T}).  This extended earlier conjectures of \cite{FIK}, which were established in \cite{SW1, CT}.
 The theory of the weak K\"ahler-Ricci flow was developed in \cite{SoT3}, where it was shown that the flow can be continued through these birational transformations at the level of potentials.  These surgeries are conjectured \cite{SoT3} to be continuous in the Gromov-Hausdorff topology, and this was confirmed for K\"ahler surfaces \cite{SW2, SW3, SW4}.  Indeed it was shown  that if the total volume is uniformly bounded below away from zero at a finite time singularity, then the K\"ahler-Ricci flow converges smoothly outside finitely many distinct holomorphic spheres with  self-intersection number $(-1)$,  and  these $(-1)$ curves are contracted to distinct points in the Gromov-Hausdorff sense. Furthermore,  making use of the work of \cite{SoT3}, it was shown that 
the flow can be extended on the blown-down surfaces, continuously in the Gromov-Hausdorff topology, and uniquely at the level of potentials.
These geometric blow-down procedures terminate in finite time. Higher dimensional metric surgeries via the K\"ahler-Ricci flow are constructed for certain families of projective manifolds in \cite{SY, S} using a construction of birational cobordism in \cite{W}.

Geometric convergence of the K\"ahler-Ricci flow at infinite time is still largely open.
 It is conjectured in \cite{SoT3} that if the K\"ahler-Ricci flow exists for all time, the normalized solution will either converge to a K\"ahler-Einstein metric space  with possible singularities (see \cite{EGZ}) or to a twisted K\"ahler-Einstein metric space of lower dimension, in the Gromov-Hausdorff topology. In particular, such spaces are expected to be homeomorphic to the canonical models of the original manifolds. This is known to hold if the manifold $X$ admits a K\"ahler-Einstein metric  \cite{C} or if there is a holomorphic submersion $\pi: X \rightarrow B$ with $c_1(B)<0$ and  smooth Calabi-Yau fibers  (established in \cite{TWY}, building on earlier work of 
 \cite{SoT1, SW3, GTZ, FZ, G}).  For more general K\"ahler manifolds for which the flow exists for all time, weak analytic convergence is proved in \cite{Ts, TZo} when $X$ is of general type and in \cite{SoT1, SoT2} when the flow collapses at infinity.  In both cases, the scalar curvature of the normalized solutions of the K\"ahler-Ricci flow  is always uniformly bounded \cite{Z2, SoT4}. However, the problem of  global geometric convergence is still open. One of the major obstacles is to obtain a uniform diameter bound for the  normalized K\"ahler-Ricci flow.

In this paper, we will consider minimal surfaces of general type. By definition, a minimal surface of general type is a smooth complex surface $X$ whose canonical bundle $K_X$ is big and nef. Equivalently, there exists a smooth closed $(1,1)$-form $\chi$ in the first Chern class $c_1(X)$ such that $\chi\geq 0$ and $\int_X \chi^2>0$. In particular, $X$ is projective and so it is K\"ahler. Since there are no $(-1)$ curves on such surfaces, the K\"ahler-Ricci flow has long time existence and it is proved in \cite{TZo} that the normalized solution will converge weakly to the pullback of the unique orbifold K\"ahler-Einstein metric on the canonical model of $X$. More precisely, we will consider the pluricanonical system of $X$.  For sufficiently large $m$, the holomorphic sections of $K_X^m$ induce a holomorphic map 
$\Phi: X \rightarrow \mathbb{P}^N$ with image $X_{\textrm{can}}$, the canonical model of $X$.  $X_{\textrm{can}}$ is an algebraic surface with, at worst, finitely many orbifold $A$-$D$-$E$-singularities (cf. \cite{BHPV}).  Moreover, $X_{\textrm{can}}$ admits a unique orbifold K\"ahler-Einstein metric \cite{Kob} because $K_{\Xc}$ is ample, which we will denote by $\oke$. The map $\Phi$ contracts $(-2)$ curves on $X$ to orbifold points on $X_{\textrm{can}}$.

We consider a solution $\omega(t)$ to  the normalized K\"ahler-Ricci flow on $X$ starting at any K\"ahler metric $\omega_0$.  Namely, $\omega=\omega(t)$ solves
\begin{equation}\label{eqn:KRF}\ddt{} \omega = - \textrm{Ric}(\omega) - \omega, \quad \omega|_{t=0}= \omega_0.\end{equation}
A unique  solution to (\ref{eqn:KRF}) exists for all time \cite{C, Ts}.   It is known by the work of Tsuji \cite{Ts} and Tian-Zhang \cite{TZo} that as $t\rightarrow \infty$, the solution $\omega(t)$ converges in $C^{\infty}$ on compact subsets of $X \setminus D$ to $\Phi^*\oke$, where $D$ is the union of the $(-2)$ curves on $X$.

Our result concerns the global Gromov-Hausdorff behavior of $\omega(t)$ as $t\rightarrow \infty$.   When  $X$ only contains disjoint irreducible $(-2)$-curves, it was shown in  \cite{SW4}  that the evolving metrics of the normalized K\"ahler-Ricci flow converge to the unique K\"ahler-Einstein orbifold metric on $\Xc$ in the Gromov-Hausdorff topology. However, when the $(-2)$ curve is not irreducible, the local metric models for the corresponding contracted $A$-$D$-$E$  singularities are usually complicated and not explicit, and so the arguments in \cite{SW2, SW4} cannot be immediately applied. 

The following is our main result.  We establish the global geometric convergence for the normalized K\"ahler-Ricci flow on all minimal surfaces of general type starting with any initial K\"ahler metric.

\begin{theorem} \label{mainthm} Let $X$ be a minimal surface of general type and $\omega(t)$ be the solution of the normalized K\"ahler-Ricci flow (\ref{eqn:KRF}). As $t \rightarrow \infty$, the K\"ahler manifolds $(X, \omega(t))$ converge in the Gromov-Hausdorff sense to $(X_{\emph{can}}, \omega_{\emph{KE}})$. In particular, the diameter of $(X, \omega(t))$ is uniformly bounded for all $t$ and the convergence is smooth on the smooth part of $X_{\emph{can}}$. 

\end{theorem}

The local smooth convergence was already proved in \cite{Ts, TZo} away from the $(-2)$ curves on $X$. Theorem \ref{mainthm} implies that all the $(-2)$ curves on $X$ are contracted to orbifold points by the K\"ahler-Ricci flow as $t\rightarrow \infty$. More precisely, if $E$ is a connected $(-2)$ curve on $X$, then for any $\varepsilon>0$, there exists an open neighborhood $U_{E, \varepsilon} $ of $E$ in $X$ such that the diameter of $U_{E, \varepsilon}$ with respect to the evolving metrics $\omega(t)$ is less than $\varepsilon$ for all sufficiently large $t$. Our approach does not rely on the a priori existence of the limiting orbifold K\"ahler-Einstein metric on $X_{\textrm{can}}$ constructed in \cite{Kob}. 

A surface of general type is a complex surface whose minimal model is a minimal surface of general type. Therefore, a surface of general type can be obtained by blowing up a minimal surface of general type finitely many times. Combined with the results of \cite{SW2, SW4}, we obtain:

\begin{corollary} Let $X$ be a general type surface. Then the normalized K\"ahler-Ricci flow on $X$ starting with any initial K\"ahler metric $g_0$ is continuous through finitely many contraction surgeries in the Gromov-Hausdorff topology for $t\in [0, \infty)$ and converges  in the Gromov-Hausdorff topology to $(X_{\emph{can}}, g_{\emph{KE}})$. The convergence is smooth away from the $(-2)$-curves, where $g_{\emph{KE}}$ is the unique orbifold K\"ahler-Einstein metric on $X_{\emph{can}}$.

\end{corollary}

We refer the readers to \cite[Theorem 1.1]{SW4}   for details of the canonical surgical contractions by the K\"ahler-Ricci flow. 

As this paper was nearing completion, the authors became aware of a preprint of Gang Tian and Zhenlei Zhang \cite{TZl} where  the $L^p$ Cheeger-Colding-Tian theory is used to obtain related results.   We remark that the estimates in our proof do not require the a priori existence of the orbifold K\"ahler-Einstein metric on the canonical model, and we only  
use the maximum principle and other elementary arguments.

The authors thank the referee for a helpful comment on a previous version of this paper.

\section{A priori estimates} 

Let $X$ be a minimal surface of general type. Since $K_X$ is big and nef, the pluricanonical linear system $|mK_X|$ is base point free for sufficiently large $m\in \mathbb{Z}^+$ and it induces a birational holomorphic map $\Phi_m = |mK_X|: X \rightarrow X_{\textrm{can}}$. The sequence $\Phi_m $ stabilizes for sufficiently large $m$ and $\Phi= \Phi_m$ contracts all the $(-2)$-curves on $X$. In particular, $K_X = \Phi^* K_{X_{\textrm{can}}}$.

We first reduce the K\"ahler-Ricci flow to a parabolic complex Monge-Amp\`ere equation, as in \cite{Ts, TZo, SW2}.  
Define a closed $(1,1)$ form $\chi$ on $X$ by
$$\chi = \frac{1}{m}\Phi^*\omega_{\textrm{FS}} \in c_1(K_X),$$
where  $\Phi: X \rightarrow \Xc \subset \mathbb{P}^N$ is the map induced by holomorphic sections of $K_X^m$, and $\omega_{\textrm{FS}}$ is the Fubini-Study metric on $\mathbb{P}^N$. 
Choose $\Omega$ to be the smooth volume form satisfying 
$$\chi = \ddbar \log \Omega, \quad \int_X \Omega = \int_X \omega_0^2$$ 
Define a family of reference K\"ahler metrics $\hat{\omega}=\hat{\omega}(t)$ on $X$ by $$\hat\omega = \chi + e^{-t}(\omega_0 - \chi).$$
Then the K\"ahler-Ricci flow \eqref{eqn:KRF} is equivalent to the parabolic complex Monge-Amp\`ere equation for $\varphi=\varphi(t)$ given by 
\begin{equation}\label{eqn:MAKRF}
\begin{aligned}
&\frac{\partial \varphi}{\partial t} = \log \frac{(\hat\omega + \ddbar \varphi)^2}{\Omega} - \varphi, \quad
&\varphi|_{t=0} = 0.
\end{aligned}
\end{equation}
Namely, if $\varphi$ solves (\ref{eqn:MAKRF}) then $\omega(t) := \hat{\omega}+\ddbar \varphi$ solves \eqref{eqn:KRF}. Conversely if we are given a solution to \eqref{eqn:KRF} then we can obtain a solution $\varphi=\varphi(t)$ of (\ref{eqn:MAKRF}).

Let $D= \sum_{i=1}^M D_i$ be the $(-2)$ curves contracted by $\Phi$, where $D_i$ is an irreducible component of $D$ and $D_i$ is a smooth rational curve with $D_i\cdot D_i = -2$. All components of $D$ meet transversally with each other (see \cite{BHPV}, for example). Since the support of $D$ is the exceptional locus of the pluricanonical map $\Phi$ and $K_X$ is big and nef, the Kodaira lemma implies that there exist $a_i \in \mathbb{Q}^+$ for $i=1, ..., M$ such that 
the $\mathbf{R}$-divisor $$K_X - \varepsilon D' $$ is ample for all $\varepsilon \in (0,1)$, where $D'= \sum_{i=1}^M a_i D_i$.

Choose $\sigma_i\in\mathcal O(D_i)$, a holomorphic section of the line bundle $[D_i]$ associated to the divisor $D_i$, vanishing to order 1 along $D_i$.  We fix now  $\varepsilon_0 \in (0, 1)$.  Then  there exist  Hermitian metrics $h_i$ on $[D_i]$ such that
\begin{equation} \label{chime}
\chi - \varepsilon_0 R_h \ge c_{\varepsilon_0} \, \omega_0,
\end{equation}
for some positive constant $c_{\varepsilon_0}$, 
where we are  using the notation $$R_h := -\ddbar \log h= \sum_{i=1}^M -a_i\ddbar \log h_i,$$ for the curvature form of the Hermitian metric $h= \otimes_i (h_i)^{a_i}$ on the $\mathbf{Q}$-line bundle associated to $[D'] = \sum_{i=1}^M a_i[D_i]$.

In the next lemma we gather together some a priori estimates which are already known by the work of Tsuji, Tian-Zhang and Zhang \cite{Ts, TZo, Z1, Z2}.

\pagebreak[3]
\begin{lemma}\label{lemma 1}  There exist uniform constants $C>0$ and $\lambda > 0$ such that the following hold.
\begin{enumerate}[label=(\roman*)]
\item\label{item 1} $\displaystyle{| \varphi| + | \dot{\varphi}| + |R| \le C}$ on $X \times [0,\infty)$.
\medskip

\item\label{item 2} On $X \times [0,\infty)$, we have $$\omega(t) \le C|\sigma|_h^{-2\lambda} \omega_0,$$ where $|\sigma|_h^2 = \prod_i |\sigma_i|_{h_i}^{2}$.
\item For any compact $K\subset X\backslash D$, there exists for every $\ell=0,1,2, \ldots$  
a constant $C_{K,\ell}$ such that
$$\|\varphi\|_{C^{\ell}(K)}\le C_{K,\ell}.$$
\end{enumerate}
\end{lemma}
\begin{proof}  Part (i) is due to Tian-Zhang and Zhang \cite{TZo, Z1, Z2}, and makes use of  the pluripotential theory estimates of Ko{\l}odziej \cite{Kol} (for a recent survey, see \cite{PSoS}).  Part (ii) was shown by Tsuji \cite{Ts}.   Part (iii) can be proved using a Calabi-type third order estimate and then higher order estimates (see \cite{PSeS} and \cite{ShW}, for example).\end{proof}

Next, we discuss the local model for an irreducible $(-2)$-curve in $X$, following the discussion in \cite{SW4}.  Fix an irreducible $(-2)$-curve $D_j$ in $X$.   Then there exists a holomorphic map $\Psi: X \rightarrow \mathbb{P}^K$ contracting $D_j$ whose image $Y:= \Psi(X) \subset \mathbb{P}^K$ is subvariety with single orbifold point $y=\Psi(D_j) \in Y$.  Using the notation above, $h_j$ is a Hermitian metric on the line bundle $[D_j]$ and $\sigma_j$ is a holomorphic section of this line bundle, vanishing to order 1 along $D_j$.  

Write $B \subset \mathbb{C}^2$ for the unit ball centered at the origin.  Write $\tilde{B} = B/\mathbb{Z}_2$, where $\mathbb{Z}_2$ acts on $B$ by
$$(z^1, z^2) \mapsto (-z^1, -z^2).$$
We can then identify a neighborhood of the orbifold point $y \in Y$ with $\tilde{B}$.  Write $\oeuc$ for the Euclidean metric
$$\omega_{\mathrm{Eucl}} = \sqrt{-1} dz^1\wedge d\bar z^1 + \sqrt{-1}dz^2\wedge d\bar z^2$$
on $B$, which descends to give a smooth  orbifold metric on $\tilde{B}$.  Define $r$ by
$$r  = \sqrt{|z^1|^2 +  |z^2|^2},$$
on $B$ (and also on $\tilde{B}$).  The map $\Psi$ is a biholomorphism from $\Psi^{-1}(\tilde{B} \setminus \{0\})$ to $\tilde{B} \setminus \{0 \}$.  The curve $D_j$ is given by $\{\sigma_j=0\}$ and we have (cf. \cite[Section 2]{SW4})
\begin{equation} \label{r4}
|\sigma_j|^2_{h_j} = r^4,
\end{equation}
where we are identifying $r$ with its pull-back via $\Psi$.  In what follows we will, without comment, identify via $\Psi$ the sets $\Psi^{-1}(\tilde{B} \setminus \{0\})$ and $\tilde{B} \setminus \{0 \}$, and similarly for the various functions and $(1,1)$ forms on these sets.  Recalling that $\omega_0$ is a fixed K\"ahler metric on $X$, we note that on $\tilde{B} \setminus \{0\}$ we have
\begin{equation} \label{simplebd}
C^{-1} r^2 \oeuc \le \omega_0 \le \frac{C}{r^2} \oeuc.
\end{equation}
Moreover, if we define the vector field $V$ on $B \setminus \{0 \}$ by
\begin{equation} \label{defineV}
V= z^1 \frac{\partial}{\partial z^1} + z^2  \frac{\partial}{\partial z^2}
\end{equation}
then we have
\begin{equation} \label{Wbounds}
\frac{r^4}{C} \le |V|^2_{\omega_0} \le Cr^4,
\end{equation}
for a uniform $C>0$.

A final remark about the local model is that if $p$ is any point in $D_j$  then we can find holomorphic coordinates $u,v$ in a neighborhood $W$ of $p$, and centered at $p$, such that $D_j\cap W= \{u=0 \}$ and
the map $\Psi$ from $W$ to $\tilde{B}$ is given by
$(z^1, z^2): W \rightarrow \tilde{B}$ where
\begin{equation} \label{uv}
u=(z^1)^2, \quad v=\frac{z^2}{z^1}.
\end{equation}
Namely,  $z^1=u^{1/2}$ and $z^2 = u^{1/2}v$, which gives a well-defined map into $\tilde{B}$.

Our next lemma shows that we have estimates for the evolving metric $\omega(t)$ on $\tilde{B} \setminus \{ 0 \}$.  However, unlike in \cite{SW4}, our $(-2)$-curve $D_j$ may intersect other $(-2)$ curves.  To begin with, we only obtain estimates away from these intersection points.  To make this more precise, fix any open tubular neighborhood $U^j_{\eta}$ of radius $\eta>0$, with respect to $\omega_0$, in $X$ of $\bigcup_{i\neq j} D_i$.  This open set $U^j_{\eta}$ may  intersect $\tilde{B} \setminus \{ 0 \}$.

We have the following estimates, which are analogous to those of Lemma 4.1 and Lemma 10.1 in \cite{SW4}. 

\begin{lemma} \label{lemmaestimates}  Fix $j \in \{ 1, \ldots, M\}$ and $\eta>0$.   
With the notation above, there exist constants $C$ and $\delta>0$ such that on $(\tilde{B} \setminus \{ 0 \}) \cap (X - U^j_{\eta})$, depending only on the fixed initial data and $\eta>0$ such that
\begin{enumerate}
\item[(i)] $\displaystyle{\omega \le \frac{C}{r^2} \omega_{\emph{Eucl}}}$.
\item[(ii)] $\displaystyle{\omega \le \frac{C}{r^{2(1-\delta)}} (\omega_0 + \omega_{\emph{Eucl}})}$.
\smallskip

\item[(iii)] $\displaystyle{|V|^2_{\omega} \le Cr^{4/3}}$, where $V$ is the vector field defined by (\ref{defineV}). 
\end{enumerate}

\end{lemma}

\begin{proof}
For simplicity of notation, we take $j=1$.
The proof is fairly straightforward, given the arguments of \cite{SW2, SW4}.  The idea is to modify the quantities in \cite{SW4} by adding a term which tends to negative infinity along the union of $D_2, \ldots, D_M$.  After applying the maximum principle, we will obtain a bound away from $U^1_{\eta}$.  

Define on $\tilde{B} \setminus \{ 0 \}$, for $\ve>0$ and $A>>1$ a large constant to be determined,
$$H_{\ve} = \log \tr{\omega_0}{\omega}+ A \log \left( |\sigma_1|_{h_1}^{(1+\ve)} \left( \prod_{i=2}^M | \sigma_i|^{4\lambda}_{h_i} \right)\tr{\oeuc}{\omega} \right) - A^2 \varphi,$$
 where $\lambda$ is the constant in item \ref{item 2} of  Lemma \ref{lemma 1}.  We wish to use the maximum principle to bound $H_{\ve}$ from above on $\tilde{B} \setminus \{0 \}$.  Using Lemma \ref{lemma 1}.(ii) and the fact that the curve $D_1$ is completely contained inside $\tilde{B}$, we see that
  the quantity $H_{\ve}$ is uniformly bounded from above on the boundary of $\tilde{B}$.   
In addition, we see from (\ref{simplebd}) that $H_{\ve}$ tends to $-\infty$ as $r$ tends to $0$, and $H_{\ve}$ tends to $-\infty$ on any curves $D_2, \ldots, D_M$ which intersect $\tilde{B}$.  We also have an upper bound for $H_{\ve}$ at time $t=0$.   We may suppose then that $H_{\ve}$ attains a maximum at an interior point $(x_0, t_0) \in (\tilde{B} \setminus \{ 0\}) \times (0,\infty)$, and $x_0$ does not lie on any of the curves $D_2, \ldots, D_M$.  It suffices to bound $H_{\ve}$ at this point $(x_0, t_0)$.
 
 First recall the following well-known pointwise differential inequality for a solution $\omega=\omega(t)$ of the K\"ahler-Ricci flow \cite{Y, A, C} (or see \cite[Prop. 3.2.5]{SW3} for a recent exposition)
 
\begin{equation} \label{di}
\left( \frac{\partial}{\partial t} - \Delta \right) \log \tr{\beta} \omega \le C_{\beta} \tr{\omega}{\beta} - 1,
\end{equation}
 for a fixed metric $\beta$, where $-C_{\beta}$ is the infimum of the bisectional curvature of $\beta$.  In particular, if the curvature of $\beta$ vanishes then we can take $C_{\beta}=0$.
 
 At $(x_0, t_0)$, we have, applying (\ref{di}) twice, first with $\beta=\omega_0$ and then with $\beta=\oeuc$, 
\begin{align*}
0 \le {} & (\frac{\partial}{\partial t}  - \Delta ) H_\varepsilon \\ \le {}  & C_0 \tr{\omega}\omega_0 + A\left( \frac{1+\varepsilon}{2} \right) \tr{\omega}{R_{h_1}} + 2A\lambda \tr{\omega}{\left(\sum_{i=2}^M R_{h_i}\right)} \\  & - A^2\dot\varphi 
 + 2A^2  - A^2 \tr{\omega}{\hat \omega}\\
= {} & C_0 \tr{\omega}\omega_0   - A^2\dot\varphi + 2A^2 \\
& - A^2 \tr{\omega}{\Big(\hat\omega - A^{-1}\left( \frac{1+\varepsilon}{2}\right) R_{h_1} - 2A^{-1}\lambda \sum_{i=2}^M R_{h_i} \Big)},
\end{align*}
where $-C_0$ is the lower bound for the bisectional curvature of $\omega_0$.
Using (\ref{chime}) and the definition of $\hat{\omega}$, we may choose $A$ uniformly large enough such that, for a uniform $c_0>0$,  $$\hat\omega - A^{-1}\left( \frac{1+\varepsilon}{2}\right) R_{h_1} - 2A^{-1}\lambda \sum_i R_{h_i}\ge c_0 \omega_0$$
and $A^2 c_0\ge C_0+1$.  Hence, using Lemma \ref{lemma 1}, part (i), we have
$$0 \le - \tr{\omega}{\omega_0} - A^2 \dot{\varphi} + CA^2,$$
for a uniform constant C.  Now recall from Lemma \ref{lemma 1} that $\varphi$ and $\dot{\varphi}$ (and hence also $\log (\omega^2/\omega_0^2)$) are uniformly bounded.  It follows that we obtain a uniform upper bound of $\tr{\omega}{\omega_0}$ at $(x_0, t_0)$, and from the inequality $\tr{\omega_0}{\omega} \le (\tr{\omega}{\omega_0}) \frac{\omega^2}{\omega_0^2}$, we also get an upper bound for $\tr{\omega_0}{\omega}$.  Then $H_{\ve}$ is bounded from above at $(x_0, t_0)$, and hence has a uniform upper bound 
on $(\tilde{B}\setminus \{ 0 \}) \times [0, \infty)$.  Moreover, our upper bound is independent of $\ve$.

Letting $\varepsilon\to 0$, we obtain on $\tilde{B}\setminus \{ 0 \})$,
\begin{equation}\label{Heb}
 \Big(\tr{\omega_0}{\omega}\Big)^{1/A} |\sigma_1|_{h_1} \left( \prod_{i=2}^M |\sigma_i|_{h_i}^{4\lambda} \right) \tr{\oeuc}{\omega}  \le C.
\end{equation}
 First, using (\ref{simplebd}) we have $|\sigma_1|_{h_1} \tr{\oeuc}{\omega} \le C \tr{\omega_0}{\omega}$ and hence
 $$\tr{\oeuc}{\omega} \le \frac{C}{r^2} \left( \prod_{i=2}^M |\sigma_i|_{h_i}^{4\lambda} \right)^{-A/(A+1)}.$$
where we have used $|\sigma_1|_{h_1}=r^2$ from (\ref{r4}).  We then obtain the estimate (i) since the quantity $\displaystyle{\Big(\prod_{i=2}^M |\sigma_i|^2_{h_i} \Big)^{-1}}$ is uniformly bounded on $X \setminus U^1_{\eta}$, depending only on the fixed data and $\eta$.
 
For (ii), observe that from (\ref{Heb}) we also have
$$\Big(\tr{\omega_0 + \oeuc}{\omega}\Big)^{1+1/A}\le C |\sigma_1|^{-1}_{h_1} \Big(\prod_{i=2}^M |\sigma_i|^{4\lambda}_{h_i} \Big)^{-1},$$
and so using (\ref{r4}) again,
 $$\tr{\omega_0 + \oeuc} \omega \le C r^{-2(1-\delta)} \Big(\prod_{i=2}^M |\sigma_i|^{4\lambda}_{h_i} \Big)^{- \frac{A}{A+1}},$$
for some $\delta\in (0,1)$ depending only on $A$.  This completes the proof of (ii).

For (iii), we will apply the maximum principle to $G_{\ve}$, for $\ve>0$, defined by 
$$G_{\ve} = \log \left( | V|_{\omega}^{1+\ve} \left( \prod_{i=2}^M | \sigma_i|^{4\lambda}_{h_i} \right) \tr{\oeuc}{\omega} \right) + A(|z^1|^2+|z^2|^2),$$
on $\tilde{B} \setminus \{ 0\}$, for $A>0$ to be determined.  Note that from the bounds (\ref{r4}), (\ref{simplebd}) and (\ref{Wbounds}), we see that for each fixed $t$,  the quantity $G_{\ve}(x)$ tends to $-\infty$ as $x$ approaches zero.  Suppose that $G_{\ve}$ achieves a maximum at a point $(x_0, t_0) \in (\tilde{B} \setminus \{0\}) \times (0,\infty)$ such that $x_0$ does not lie on any of the curves $D_2, \ldots, D_M$.

Following the computation of \cite[Lemma 2.6]{SW2}, we compute at $(x_0, t_0)$,
$$0 \le (\frac{\partial}{\partial t} - \Delta) G_\varepsilon \le   \tr{\omega}{\Big( \sum_{i=2}^M 2\lambda R_{h_i} - A \oeuc \Big)} < 0,$$
if $A$ is chosen sufficiently large.  This is a contradiction, implying that the maximum of $G_{\ve}$ occurs at $t=0$ or on the boundary of $\tilde{B}$, where, by Lemma \ref{lemma 1}, we have the upper bound for $G_{\ve}$.    Thus $G_{\ve}$ is bounded from above on $\tilde{B} \setminus \{0\}$, and moreover the bound is independent of $\ve$.  Letting $\ve \rightarrow 0$, we obtain
\begin{equation*}
|V|_\omega \tr{\oeuc} \omega\le C \Big(\prod_{i = 2}^M |\sigma_i|^{4\lambda}_{h_i} \Big)^{-1}.
\end{equation*}
But on $(X \setminus U^1_{\eta} \cap (\tilde{B} \setminus \{0\})$, we have from (i),
\begin{equation*}
 |V|_\omega \tr{\oeuc} \omega\le C.
\end{equation*}
On the other hand, $$|V|_\omega^2\le (\tr{\oeuc}{\omega}) |V|_{\oeuc}^2 = r^2 \tr{\oeuc}{\omega},$$
and multiplying both sides by $|V|_\omega$ gives
 $|V|_\omega^2\le C r^{4/3}$ on $(X \setminus U^1_{\eta} \cap (\tilde{B} \setminus \{0\})$.  This completes the proof of the lemma.
 \end{proof}

As an immediate consequence of these estimates, we can obtain the following bound on the evolving metric $\omega=\omega(t)$ in terms of local holomorphic coordinate systems on $X$.

\begin{corollary} \label{cortolemma}  Fix  $j \in \{1, \ldots, M\}$ and $\eta>0$.  We can cover $D_j \cap (X \setminus U^j_{\eta})$ by finitely many complex coordinate charts $W$ with complex coordinates $u,v$ so that 
\begin{enumerate}
\item[(i)] On each $W$, the curve $D_j$ is given by $D_j= \{ u=0 \}$.
\item[(ii)] There exists a uniform $C$ such that on each $W$,
\begin{equation} \label{goode}
\omega(t) \le C \left( \frac{ \sqrt{-1} du \wedge d\ov{u}}{|u|^{4/3}} + \sqrt{-1} dv\wedge d\ov{v} \right).
\end{equation}
\end{enumerate}
\end{corollary}
\begin{proof}
By compactness, it suffices to prove the result for some neighborhood of a fixed point $p$ in  $D_j \cap (X \setminus U^j_{\eta})$.  By the discussion above the statement of Lemma \ref{lemmaestimates}, we can find coordinates $u,v$ in a neighborhood $W$ of $p$ such that $W \cap D_j = \{u=0\}$ and $u,v$ satisfy (\ref{uv}).  Compute on $W\backslash D_j$
\[
\begin{split}
\sqrt{-1} (dz^1 \wedge d\ov{z}^1 + dz^2 \wedge d\ov{z}^2) = {} & \frac{(1+|v|^2)\sqrt{-1} du \wedge d\ov{u}}{4|u|} + \sqrt{-1} |u| dv\wedge d\ov{v} \\
& + \textrm{Re} \left( \sqrt{-1} \left( \frac{\ov{u}}{u} \right)^{1/2} v du \wedge d\ov{v} \right) \\
\le {} & C \left( \frac{ \sqrt{-1} du \wedge d\ov{u}}{|u|} +  \sqrt{-1} |u| dv\wedge d\ov{v}\right),
\end{split}
\]
if we assume that $v$ is sufficiently small.  Moreover, we have \begin{equation} \label{r2}
r^2 = |z^1|^2+|z^2|^2= |u|(1+|v|^2).
\end{equation}
Then the estimate (i) of Lemma \ref{lemmaestimates} implies that on $W$ we have
$$\omega \le C\left( \frac{ \sqrt{-1} du \wedge d\ov{u}}{|u|^2} + \sqrt{-1} dv\wedge d\ov{v} \right),$$
for a uniform constant $C$.

We can improve the bound in the $\partial/\partial u$ direction.  Indeed, note that 
$$u \frac{\partial}{\partial u} =\frac{ z^1}{2} \frac{\partial}{\partial z^1} + \frac{z^2}{2} \frac{\partial}{\partial z^2} =\frac V 2,$$
and hence from part (iii) of Lemma \ref{lemmaestimates} and (\ref{r2}),
$$g\left( \frac{\partial}{\partial u},  \frac{\partial}{\partial \ov{u}} \right) = \frac{1}{4 |u|^2}|V|^2_{\omega} \le \frac{Cr^{4/3}}{|u|^2} \le \frac{C'}{|u|^{4/3}},$$
and this gives the estimate (ii) on $W$, as required.  
\end{proof}

Note that although the estimate (\ref{goode}) blows up as we approach $\{ u= 0 \}$, it is strong enough to show that the diameter of each set $U_{\eta}^j$ is uniformly bounded. However at the moment, we have no estimates in neighborhoods of the  intersection points $p_{ij}=D_i \cap D_j$ for $i\neq j$.   The following lemma shows that our estimates do indeed extend to such neighborhoods.

\begin{lemma}  \label{lemmacon} Fix an intersection point $p_{ij} = D_i \cap D_j$ for $i\neq j$.  There exists a complex coordinate chart $W$ centered at $p_{ij}$ with complex coordinates $u,v$, such that $D_i \cap W = \{ u=0 \}$ and $D_j \cap W = \{ v=0 \}$.  On $W \setminus (D_i \cup D_j)$, we have
\begin{equation} \label{keyestimate}
\omega(t) \le C \left( \frac{\sqrt{-1} du \wedge d\ov{u}}{|u|^{4/3}} +  \frac{\sqrt{-1} dv \wedge d\ov{v}}{|v|^{4/3}}\right).
\end{equation}
for a uniform constant $C$.
\end{lemma}
\begin{proof}  Since $D_i \cdot D_j=1$ if $i\neq j$ and $D_i$ and $D_j$ meet transversally, we can always find local complex coordinate chart $W$ centered at $p_{ij}$ with complex coordinates $u,v$, such that $D_i \cap W = \{ u=0 \}$ and $D_j \cap W = \{ v=0 \}$.

Write
$$\ocon = \frac{\sqrt{-1} du \wedge d\ov{u}}{|u|^{4/3}} +  \frac{\sqrt{-1} dv \wedge d\ov{v}}{|v|^{4/3}}.$$
Note that $\ocon$ is a metric with conical singularities with simple normal crossings. The cone angle along $(D_i \cup D_j)\setminus \{p_{ij}\}$ is $2\pi/3$ (cf. \cite{D, DGSW}).  A key point is that on $W \setminus (D_i \cup D_j)$, the curvature of $\ocon$ vanishes, as can be readily checked by a direct computation, or alternatively, one can apply change of variables by letting $u' = u^{1/3}$, $v' = v^{1/3}$ locally on a well-defined branch.

We will apply the maximum principle to the quantity $G = \log \tr{\ocon}{\omega}$ on $W \setminus (D_i \cup D_j)$.  The estimates of Corollary \ref{cortolemma} imply that the quantity $G$ is uniformly bounded from above on the boundary of the set $W$.  Applying (\ref{di}) with $\beta = \ocon$ and $C_0=0$ we see that on $W \setminus (D_i \cup D_j)$ we have
$$\left( \frac{\partial}{\partial t} - \Delta \right)  G\le -1,$$
which implies that the maximum of $G$ occurs either at $t=0$ or on the boundary of $W$.  In both cases, we have a uniform upper bound\footnote{To clarify this part of the argument (we thank Valentino Tosatti for pointing out our imprecision here): fix $\ve>0$ and consider $G_{\ve} = \log \tr{\ocon}{\omega} + \ve \log \frac{\omega^2}{\ocon^2}$ which satisfies $( \frac{\partial}{\partial t} - \Delta )  G_{\ve} \le -1$.  Then let $\ve \rightarrow 0$.}  for $G$.  Hence $G \le C$ on $W \setminus (D_i \cup D_j)$, completing the proof of the lemma.
\end{proof}

From the estimates now obtained, we can bound distances with respect to $g(t)$ in a neighborhood of the curve $D$.  Denote by $B_{g_0, \delta}(p)$ the $g_0$ geodesic ball in $X$ centered at $p \in X$ with radius $\delta$, and by $\tilde{U}_{\delta}$ the $\delta$-tubular neighborhood of $D$ with respect to $g_0$,
\begin{equation} \label{tildeU}
\tilde{U}_{\delta}:= \{ x \in X \ | \ d_{g_0} (x, D) < \delta \}.
\end{equation}
Then we have:

\begin{proposition} \label{prop} There exist uniform positive constants $C, \delta_0$ so that the following hold.
\begin{enumerate}
\item[(i)] For every intersection point $p_{ij} = D_i \cap D_j$, and any $\delta\in (0, \delta_0)$, we have
$$\emph{diam}_{g(t)} B_{g_0, \delta}(p_{ij}) \le C \delta^{1/3}.$$
\item[(ii)]  For every $\delta \in (0, \delta_0)$, 
$$d_{g(t)}(D, q) \le C \delta^{1/3}, \quad \textrm{for all } q  \in \tilde{U}_{\delta}.$$
\end{enumerate}
\end{proposition}
\begin{proof} For (i), fix an intersection point $p_{ij}$, and let $W$ be as in the statement of Lemma \ref{lemmacon}, with $p_{ij}$ corresponding to the point $u=v=0$ in the complex coordinates $u,v$.  We assume that $\delta_0>0$ is sufficiently small so that $W$ contains the ball $B_{g_0, \delta_0}(p_{ij})$.  Given a point $(u_0, v_0) \in W$ with both $u_0$ and $v_0$ nonzero, we compute using the estimate (\ref{keyestimate}) that the length of the path $\gamma: [0,1] \rightarrow W$ given by $\gamma(\lambda) = (\lambda u_0, \lambda v_0)$, with respect to $g(t)$, satisfies
$$\textrm{Length}_{g(t)}(\gamma) \le C(|u_0|^{1/3} + |v_0|^{1/3}) \le C' \delta^{1/3},$$
for uniform constants $C, C'$.  Since for each $t$, the metric $\omega(t)$ is smooth, we obtain this bound even when one of $u_0$, $v_0$ is zero.   This gives (i).

For (ii), let $q \in \tilde{U}_{\delta}$.  First assume that $q$ is contained in a coordinate chart $W$ as in the discussion above, and $q$ has coordinates $(u_0, v_0)$.  Since $q$ is in $\tilde U_{\delta}$, at least one of $|u_0|$ or $|v_0|$ is of order $\textrm{O}(\delta)$.  Without loss of generality we suppose that $|u_0| = \textrm{O}(\delta)$ and $v_0\neq 0$.
  Then consider the path $\gamma: [0,1] \rightarrow W$ given by $\gamma(\lambda) = (\lambda u_0, v_0)$.  The length of this path with respect to $g(t)$ satisfies
$$\textrm{Length}_{g(t)}(\gamma) \le C|u_0|^{1/3} \le C' \delta^{1/3},$$
since when restricted to the $\partial/\partial u$ and $\partial/\partial \ov{u}$ directions, the metric $\omega(t)$ is bounded from above by $C\sqrt{-1}du \wedge d\ov{u}/|u|^{4/3}$. 

If $q \in \tilde{U}_{\delta}$ is not contained in such a neighborhood $W$ of an intersection point, we instead apply the estimate of Corollary \ref{cortolemma}, and we obtain the required bound by an identical argument.  This completes the proof of (ii).
\end{proof}

\section{Gromov-Hausdorff convergence}

Given the estimates above, it is now straightforward to derive the Gromov-Hausdorff convergence of $(X,g(t))$ to $(X_{\textrm{can}}, g_{\textrm{KE}})$.  First, we have
\begin{equation} \label{intestD}
\int_D \omega(t) = e^{-t} \int_D \omega_0 \longrightarrow 0, \quad \textrm{as } t \rightarrow \infty,
\end{equation}
where $D = \sum_i D_i$ is the $(-2)$ curve as introduced above.  Indeed, this follows immediately from the K\"ahler-Ricci flow equation since
$$\ddt{} \int_D \omega(t) = - \int_D \textrm{Ric}(\omega(t)) - \int_D \omega(t) = -\int_D \omega(t),$$
where we have used the fact that $K_X \cdot D=0$.

First we prove:

\begin{lemma} \label{D} If $\hat{D}$ is a connected component of $D$ then 
$$\emph{diam}_{g(t)} \hat{D} \rightarrow 0, \quad \textrm{as } t \rightarrow \infty.$$
\end{lemma}
\begin{proof}
For simplicity, assume that $\hat{D}= D_1 \cup D_2$, and that $D_1$ and $D_2$ intersect at a point $p$ (the general case is similar).
Fix $\ve>0$.  Then by Proposition \ref{prop} there exists $\delta>0$ such that the $g_0$ geodesic ball $B_{g_0,\delta}(p)$ of radius $\delta$ centered at $p$ has
\begin{equation} \label{ve1}
\textrm{diam}_{g(t)} B_{g_0, \delta}(p) < \ve/3.
\end{equation}
Define $D_1^{(\ve)} = D_1 \setminus B_{g_0, \delta}(p)$.
The estimate of Lemma \ref{lemmaestimates}, part (ii), together with \cite[Lemma 5.1]{SW4} gives the existence of a constant $C_{\ve}$, depending on $\ve$, such that
$$\omega(t)|_{D_1^{(\ve)}} \le C_{\ve} \omega_0|_{D_1^{(\ve)}}.$$
We claim that if $a, b$ are any two points in $D_1^{(\ve)}$, then 
$$d_{g(t)} (a,b) \le C'_{\ve} e^{-t/3},$$
for another constant $C'_{\ve}$ depending on $\ve$.  The claim follows almost verbatim from the argument in \cite{SW2} or \cite[Lemma 2.4]{SSW} and the estimate (\ref{intestD}).  Indeed, the only differences are that we have to keep track of the dependence on $\ve$, and we replace the factor $(T-t)$ in \cite[Lemma 2.4]{SSW} with $e^{-t}$.

The proof is now essentially finished.   Choose $t$ sufficiently large so that $C'_{\ve} e^{-t/3}< \ve/3$, so that
\begin{equation}\label{ve2}
\textrm{diam}_{g(t)} D_1^{(\ve)} \le \ve/3.
\end{equation}
Similarly, if we set $D_2^{(\ve)} = D_2 \setminus B_{g_0, \delta}(p)$ then we obtain
\begin{equation} \label{ve3}
\textrm{diam}_{g(t)} D_2^{(\ve)} \le \ve/3.
\end{equation}
The result then follows by combining (\ref{ve1}), (\ref{ve2}) and (\ref{ve3}).
\end{proof}

Combining this lemma with the estimates established in Proposition \ref{prop}, we immediately obtain:

\begin{lemma} \label{lemmatubular} Let $\hat{D}$ be a connected component of $D$. For $\delta>0$, write
$$\hat{U}_{\delta}: = \{ x \in X \ | \ d_{g_0}(x, \hat{D}) < \delta \},$$
for its $\delta$-tubular neighborhood with respect to $g_0$.  Then 
for any $\ve>0$ there exists $\delta>0$ and $T>0$ such that 
\begin{equation}\label{diamtU}
\emph{diam}_{g(t)} \hat{U}_{\delta} \le \ve, \quad \textrm{for } t\ge T.
\end{equation}

As a consequence, the diameter of $(X, g(t))$ is uniformly bounded.
\end{lemma}
\begin{proof} Combine part (ii) of Proposition \ref{prop} with Lemma \ref{D} to obtain (\ref{diamtU}).  For the last assertion, fix $\delta>0$ and use Lemma \ref{lemma 1} to see that the diameter of 
$X \setminus V_{\delta}$ is uniformly bounded, where $V_{\delta}$ is the union of the $\delta$-tubular neighborhoods $\hat{U}_{\delta}$ of all the  connected components of $D$.
\end{proof}

Finally, we complete the proof of our main result.

\begin{proof}[Proof of Theorem \ref{mainthm}] By the results in \cite{Ts} and \cite{TZo}, the solution $\varphi(t)$ of the Monge-Amp\`ere flow (\ref{eqn:MAKRF}) converges smoothly on $X\setminus D$ to a bounded $\chi$-psh function $\varphi_{\textrm{KE}}$ satisfying 
$$(\chi+\ddbar \varphi_{\textrm{KE}})^2 = e^{\varphi_{\textrm{KE}}}\Omega$$
as $t\rightarrow \infty$ and $\varphi_{\textrm{KE}} \in C^\infty(X\setminus D)$.  Since $\chi$ vanishes on $D$ and $\varphi_{\textrm{KE}}\in L^\infty(X)\cap \textrm{PSH}(X, \chi)$, $\varphi_{\textrm{KE}}$ is constant along each connected component of $D$. Immediately $ \chi+ \ddbar \varphi_{\textrm{KE}} = \Phi^* \omega_{\textrm{KE}}$ for some  K\"ahler current   $\omega_{\textrm{KE}} \in -c_1(\Xc)$ on $\Xc$ with bounded local potentials. We then can apply the smoothing properties of the weak K\"ahler-Ricci flow \cite[Theorem 4.5]{SoT2} on K\"ahler orbifolds to  show that $\omega_{\textrm{KE}}'$ must  be a smooth orbifold K\"ahler-Einstein metric on $\Xc$. In particular, $\omega_{\textrm{KE}} = \omega_{\textrm{KE}}'$ by the uniqueness of orbifold K\"ahler-Einstein metrics on K\"ahler orbifolds with ample canonical bundle. 

From the local smooth convergence of  $\varphi$ to $\varphi_{\textrm{KE}}$ on $X\setminus D$,  we have
$$\omega(t) \rightarrow \omega_{\textrm{KE}}, \quad \textrm{as } t \rightarrow \infty,$$
in $C^{\infty}(K)$ for any compact subsets $K \subset X \setminus D$.  Combining this fact with Lemma \ref{lemmatubular} and the argument of \cite[Section 3]{SW2}, 
we immediately obtain the Gromov-Hausdorff convergence of $(X, \omega(t))$ to $(X_{\textrm{can}}, \omega_{\textrm{KE}})$.
\end{proof}

\end{document}